\documentclass{amsart}

\usepackage{cite}

\usepackage{amsmath}
\usepackage{amstext}
\usepackage{amssymb}

\usepackage{amsthm}

\theoremstyle{plain}
\newtheorem{theorem}{Theorem}[section]
\newtheorem{lemma}[theorem]{Lemma}

\theoremstyle{definition}
\newtheorem{definition}{Definition}[section]
\newtheorem{example}{Example}[section]
\numberwithin{equation}{section}

\newcommand{\conj}[1]{\overline{#1}}
\DeclareMathOperator{\sgn}{sgn}

\DeclareMathOperator{\Rp}{Re}

\newcommand{\bp}{\mathcal{P}}

\begin{document}

\title[Approximation of Extremal Functions]%
{Uniform Approximation of Extremal Functions in Weighted Bergman Spaces}
\author{Timothy Ferguson}
\address{Department of Mathematics\\University of Alabama\\Tuscaloosa, AL}
\email{tjferguson1@ua.edu}
\thanks{Partially supported by RGC Grant RGC-2015-22 from the University of Alabama.}

\date{\today}

\begin{abstract}
We discuss approximation of extremal functions 
by polynomials
in the weighted Bergman spaces $A^p_\alpha$ where 
$-1 < \alpha < 0$ and $-1 < \alpha < p-2$. 
We obtain bounds on how close the approximation 
is to the true extremal function 
in the $A^p_\alpha$ and 
uniform norms.  We also discuss several 
results on the relation between the Bergman 
modulus of continuity of a function and how 
quickly its best polynomial approximants 
converge to it. 
\end{abstract}

\thanks{Thanks to Brendan Ames for a helpful 
discussion.}

\maketitle

\section{Introduction}

In this article we discuss uniform approximation of extremal functions 
in weighted Bergman spaces.  In general, we approximate these functions 
by solutions to extremal problems restricted to spaces of polynomials.  

\begin{definition}
For $1 < p < \infty$ and $-1 < \alpha < \infty$ 
we define the weighted Bergman space $A^p_\alpha$ to be the space of all 
analytic functions in $\mathbb{D}$ such that 
\[
\|f\|_{p,\alpha} = \left( \int_{\mathbb{D}} |f(z)|^p dA_\alpha(z) \right)^{1/p}
 < \infty,
\]
where $dA_\alpha = (\alpha + 1)\pi^{-1}(1- |z|^2)^\alpha \, dA(z)$ and 
$dA$ is Lebesgue measure.
\end{definition}

For $1 < p < \infty$, it is know that the dual of 
$A^p_\alpha$ is isomorphic to $A^q_\alpha$, where $1/p + 1/q = 1$. 
Also, if $\phi \in (A^p_\alpha)^*$ and $k \in A^q_\alpha$ correspond to each 
other, then 
$\|\phi\|_{(A^p_\alpha)^*} \leq \|k\|_{A^q_\alpha} 
   \leq C \|\phi\|_{(A^p_\alpha)^*}$, 
where 
$C$ is some constant depending of $p$ and $\alpha$. 

\begin{definition}
Let 
$k \in  A^q_\alpha$ be given, where $1 < q < \infty$ and 
$k$ is not identically $0$.  
Let 
$F \in A^p_\alpha$ be such that $\|F\| = 1$ and 
$\Rp \int_{\mathbb{D}} F \overline{k} \, dA_\alpha$ is as large 
as possible, where $1/p + 1/q = 1$.  
There is always a unique function $F$ with this property.
We say that $F$ is the extremal function for the integral 
kernel $k$, and also that $F$ is the extremal function for the 
functional $\phi$ defined by 
$\phi(f) = \int_{\mathbb{D}} f \overline{k} \, dA_\alpha$. 
\end{definition}
We do not usually discuss the case $p=2$ because in this case 
$F$ is a scalar multiple of $k$. 

It is known (see \cite{Clarkson}) that the spaces $A^p_\alpha$, since 
they are subspaces of $L^p$ spaces, are uniformly convex.  In 
\cite{tjf1}, general results are proven about approximating 
extremal functions in uniformly convex spaces, and a proof is 
given there of the well known fact that extremal functions are 
unique in uniformly convex spaces.  
See 
\cite{Beneteau_Khavinson_survey, 
Beneteau_Khavinson-selected_problems_classical_function} for
more information on extremal problems in spaces of 
analytic functions. 
See also \cite{Ryabykh, tjf:holder-bergman, 
Khavinson_McCarthy_Shapiro, 
Shapiro_Regularity-closest-approx} 
for more information on 
regularity questions related the extremal 
problems we discuss. 

\begin{definition}
Let $f \in A^p_\alpha$.  Suppose 
\[
\|f(e^{it} \cdot) + f(e^{-it} \cdot) - 2f(\cdot) \|_{p,\alpha} 
\le C |t|^{\beta}
\]
for some constant $C$.  
We then say that $f \in \Lambda^*_{\beta, A^p_\alpha}$.  Furthermore, we 
define $\|f\|_{\Lambda^*,\beta,A^p_\alpha}$ 
to be the infimum of the constants $C$ such 
that the above inequality holds. 
\end{definition}

We refer to functions in the $\Lambda^*$ classes 
as being (mean) Bergman-H\"{o}lder continuous 
(see \cite{tjf:holder-bergman}). 
We discuss several estimates that relate the mean Bergman-H\"{o}lder 
continuity of $A^p_\alpha$ functions to the minimum error in approximating 
these functions with polynomials of fixed degree.  We apply these results 
to obtain estimates for how close the solution of an extremal problem is 
to the solution to the problem with the same linear functional posed over 
the space of polynomials of degree at most $n$.  
By using inequalities related to 
uniform convexity 
due to Clarkson \cite{Clarkson} and Ball, Carlen and Lieb 
\cite{Ball_Carlen_Lieb}, we are able to obtain quantitative estimates for 
distance from approximate extremal functions to the true extremal 
functions.

The estimates just mentioned are all in the $A^p_\alpha$ norm.  
However, our goal is to 
approximate (in certain cases) extremal functions in the 
uniform norm (i.e. the $L^\infty$ norm). 
To do so, we use results from \cite{tjf:holder-bergman} to obtain bounds 
on the $C^{\beta}$ norm of the extremal functions and the 
functions approximating them for certain $\beta$, 
as long as the integral kernels are sufficiently regular.   
We also use
Theorem \ref{thm:holder_bergman_uniform_bound}, which allows us to 
conclude that two functions that are each not too large in the $C^\beta$ 
norm and that are close in the $A^p_\alpha$ norm must actually be 
close in the uniform norm.  In stating the theorems, we do not aim for the 
most general estimates possible; however, the estimates we state do apply to 
the case where $k$ is a polynomial, or even in $C^2(\overline{\mathbb{D}})$. 

We note that 
in \cite{Khavinson_Stessin}, Khavinson 
and Stessin derive
H\"{o}lder regularity results for 
extremal problems in unweighted Bergman spaces,
However, they do not state explicit bounds on the 
exponent $\beta$ or on the $C^\beta$ norm of the 
extremal function, so we cannot use their result 
to get explicit bounds on extremal functions.  

The following lemma about the uniform convexity of $L^p$ will be needed.  
The inequality for $1 < p \leq 2$ can be 
proved from Theorem 1 in \cite{Ball_Carlen_Lieb}. 
The other inequality follows from 
equation (3) in Theorem 2 in \cite{Clarkson}. 

\begin{lemma}\label{lemma:sharp_uniform_convexity}
Let $\|f\|_p = \|g\|_p = 1$ and $(1/2)\|f + g\|_p > 1 - \delta$.  Let 
$\|f - g\|_p = \epsilon$.  
If $1< p \leq 2$ then $\epsilon < \sqrt{\frac{8}{p-1}} \delta^{1/2}$.  If 
$p \geq 2$ then $\epsilon < 2 p^{1/p} \delta^{1/p}$.  
\end{lemma}

\section{Mean Holder Continuity and Best Polynomial Approximation}
In this section we discuss several results relating mean 
H\"{o}lder continuity of functions to their distance from the space of 
polynomials of degree at most $n$.  Some of these results are used in the 
rest of the paper.  
The proofs of these results 
are similar to the proofs for similar results about 
classical H\"{o}lder continuity that can be found in \cite{Zygmund}, Volume 1, starting on 
p.\ 115.

\begin{definition}
Let $f \in A^p_\alpha$.  We define 
\[
E_n^{p,\alpha}(f) = \min \{\|f - P\|_{p,\alpha}: \text{ $P$ is a polynomial 
of degree at most $n$}\}.
\]
\end{definition}

\begin{theorem}
Let $0< \beta < 1$.  
Suppose that $\|f\|_{\Lambda^*, \beta, A^p_\alpha} = M$.  Let 
\[
A_\beta = \frac{2^{1+\beta}}{\pi} \int_0^\infty |\cos(t) - \cos(2t)|t^{\beta - 2} \, dt.
\]
Then 
\[
E_n^{p,\alpha}
\leq A_{\beta} n^{-\beta}    \|f\|_{\Lambda^*, \beta, A^p_\alpha}.
\]
\end{theorem}
\begin{proof}
Let $M = \|f\|_{\Lambda^*, \beta, A^p_\alpha}.$ 
Let $f_{|r}$ represent the function $f$ restricted to the circle of 
radius $r$. 
Let $T_n$ be the best polynomial approximant of $f$, let 
$R_n = f - T_n$ be the remainder and let $\rho_k$ be the $k^{\text{th}}$ 
Ces\`{a}ro sum of the remainder.  
Let $K_m$ be the the Fej\'{e}r kernel for the 
$m^{th}$ Ces\`{a}ro sum.  Then $K_n$ has $L^1$ norm of $1$, and Young's 
inequality for convolutions shows that 
$M_p(r,\rho_k) = \|R_{n|r} * K_k\|_p \leq M_p(r,R_n) \|K_k\|_1 = 
M_p(r,R_n)$.  Let $\sigma_k$ be the $k^{\text{th}}$ Ces\`{a}ro sum 
of $f$.  
From \cite[eq. (13.4), p.\ 115, Volume 1]{Zygmund} we see that
\[
\left(1 + \frac{n}{h}\right) \sigma_{n+h-1} - \frac{n}{h}\sigma_{n-1} = 
T_n + \left(1 + \frac{n}{h}\right) \rho_{n+h-1} - \frac{n}{h} \rho_{n-1}.
\]
Using this equation
with $h=n$, subtracting $f$ from both sides and using
the fact that $M_p(r,\rho_k) \leq M_p(r,R_n)$ %
shows that 
$M_p\big[r,(2\sigma_{2n-1} - \sigma_{n-1}) - f\big] \leq 4 M_p(r,R_n)$.  
Multiply by $(\alpha + 1)2r(1-r^2)^\alpha$ 
and integrate $r$ from $0$ to $1$ to see that 
\[
\|(2\sigma_{2n-1} - \sigma_{n-1}) - f\|_{A^p_\alpha} \leq 
  4 \|R_n\|_{A^p_\alpha}.
\]

Let $\tau_n = 2\sigma_{2n-1}-\sigma_{n-1}$. 
Now 
\begin{equation}\label{eq:tau_m}
\tau_m(re^{ix}) - f(re^{ix}) = 
   \frac{2}{\pi} \int_0^\infty \left[ f(re^{i[x+(t/m)]}) + 
                    f(re^{i[x-(t/m)]}) - 2f(re^{ix}) \right] \frac{h(t)}{t^2} 
\, dt
\end{equation}
where $h(t) = (\cos(t) - \cos(2t))/2$. 
Apply Minkowski's inequality to see that 
\[
E^{p,\alpha}_{2m-1} \leq 
\| \tau_m(re^{ix}) - f(re^{ix}) \|_{p,\alpha} \leq 
\frac{2}{\pi} \int_0^\infty t^\beta M m^{-\beta} \frac{|h(t)|}{t^2} \, dt
= A_\beta M (2m)^{-\beta}.
\]
Since $E^{p,\alpha}_{2m} \leq E^{p,\alpha}_{2m-1}$, the 
theorem follows. 
\end{proof}

We can also prove the following theorem. 
\begin{theorem}\label{thm:approx_nderiv}
Let $K \geq 0$ be an integer. 
Suppose that $\|D^K_\theta f(re^{i\theta})\|_{p,\alpha} \leq M$.  
Let 
\[C_k = \frac{4}{\pi} \int_0^\infty |H_k(t)| \, dt\]
where 
\[
H_0(t) = h(t)/t^2, \qquad H_k(t) = \int_t^\infty H_{k-1}(x) \, dx.
\]
Then 
$E_n^{p,\alpha} \leq 2^K C_K M n^{-K}$.   
\end{theorem}
\begin{proof}
Let $\displaystyle f^{(n,\theta)}(re^{i\theta}) = 
  \frac{\partial^n}{\partial \theta^n} f(re^{i\theta})$.  
Then integrating by parts in equation \eqref{eq:tau_m} shows that 
\[
\tau(re^{ix}) - f(re^{ix}) = 
   \frac{2}{\pi m^K} \int_0^\infty \left[ 
  f^{(k,\theta)}(re^{i[x+(t/m)]}) + 
 (-1)^K f^{(K,\theta)}(re^{i[x-(t/m)]})\right] 
       H_K(t) \, dt
\]
Applying Minkowski's inequality shows that 
\[
\|\tau_m(x) - f(x)\|_{p,\alpha} \leq C_K m^{-K} \|D_\theta^n f\|_{p,\alpha}.
\]
As above, 
this implies that 
$E_n^{p,\alpha} \leq C_K 2^K M n^K$. 
\end{proof}

Define the $A^p_\alpha$ modulus of continuity for $f$ by
\[
\omega_{p,\alpha}(\delta, f) = \sup_{t \leq \delta} \|f(e^{it}z) - f(z)\|_{p,\alpha}. \]
\begin{theorem}
Let $k \geq 0$ be an integer. 
Suppose  $D_\theta^k f$ has modulus of continuity $\omega_{p,\alpha}(\delta)$.  
Then 
\[
E_n^{p,\alpha}(f) \leq B_k \omega_{p,\alpha}\left(\frac{2\pi}{n}\right) n^{-K}
\]
where $B_k = 2^{K} C_{K+1}/\pi + 2^K C_K$. 
\end{theorem}
Let 
$f_\delta(z) = \frac{1}{2\delta} \int_{-\delta}^{\delta} f(e^{it}z) \, dt$. 
Note that $D_\theta f_\delta = (D_\theta f)_\delta$. 
Minkowski's inequality shows that 
$\|f_\delta - f\|_{p, \alpha} \leq \omega_{p,\alpha}(\delta, f)$.  Let 
$f = f_\delta + g$.  Then using the fundamental theorem of calculus, we 
see that 
\[
\|D_\theta^{K+1} f_\delta\|_{p,\alpha} = 
\frac{\|D_\theta^K f(ze^{i\delta}) - D_\theta^K f(ze^{-i\delta})\|_{p,\alpha}}{2\delta}
\leq 2\delta^{-1} \omega_{p,\alpha}(2\delta; D_\theta^K f) 
\]
Also $\|D_\theta^Kg\|_{p,\alpha} \leq \omega_{p,\alpha}(\delta,D_\theta^K f)$. 
Thus by Theorem \ref{thm:approx_nderiv}, 
\[
E_n^{p,\alpha}(f) \leq 2^{K+1} C_{K+1}  n^{-(K+1)} (2\delta)^{-1} 
   \omega_{p,\alpha}(2\delta,D_\theta^K f) 
+ 2^K C_K n^{-K} \omega(\delta; D_\theta^K f).
\]
Taking the supremum over $|t| < \delta$ in the inequality
\[
\|f(\cdot)-f(e^{-2it}\cdot)\|_{p,\alpha} \leq
\|f(\cdot)-f(e^{-it}\cdot)\|_{p,\alpha} + 
   \|f(e^{-it}\cdot)-f(e^{-2it}\cdot)\|_{p,\alpha}
\]
shows that $\omega_{p,\alpha}(2\delta,f) \leq 2 \omega_{p,\alpha}(\delta,f)$. 
Thus
\[
E_n^{p,\alpha}(f) \leq 2^{K+1} C_{K+1}  n^{-(K+1)} \delta^{-1} 
   \omega_{p,\alpha}(\delta,D_\theta^K f) 
+ 2^K C_K n^{-K} \omega(\delta; D_\theta^K f).
\]
Now choose $\delta = 2\pi/n$ to see that 
\[
E_n^{p,\alpha}(f) \leq B_k \omega_{p,\alpha}\left(\frac{2\pi}{n}\right) n^{-K}
\]
where $B_k = 2^{K} C_{K+1}/\pi + 2^K C_K$.

From this it follows that if $f \in \Lambda^*_{\beta,A^p_\alpha}$ for 
$0 < \beta < 1$ then 
$E_n^{p,\alpha}(f) 
  \leq (2\pi)^\beta B_0 \|f\|_{\Lambda^*,\beta, A^p_\alpha} n^{-\beta}$. 

\begin{theorem}
Suppose that $f^{(\theta,K)} \in \Lambda^*_{1,A^p_\alpha}$ and that 
$\| f^{(\theta,K)}\|_{\Lambda^*,1,A^p_\alpha} = M$.
Then 
$E_n^{p,\alpha}(f) \leq \widetilde{B_k} M n^{-K-1}$ where 
\[
\widetilde{B_k} = 2^K (C_{k+2}/\pi + \pi C_k). 
\]
\end{theorem}
\begin{proof}
Write 
$f = f_{\delta \delta} + g$ where 
$f_{\delta \delta} = (f_\delta)_\delta$. Then 
\[
\partial_t^{K+2} f_{\delta \delta}(re^{it}) = 
\frac{f^{(\theta,K)}(re^{i(t+2\delta)}) + f^{(\theta,K)}(re^{i(t-2\delta)}) - 2 f^{(\theta,K)}(re^{it})}{4\delta^{2}}
\]
as in the last equation on \cite[Volume 1, p.\ 117]{Zygmund}. 
Thus 
\[
\|\partial_t^{K+2} f_{\delta \delta}\|_{p,\alpha}
\leq \frac{M}{2\delta}.
\]
Following the first and second equations on 
\cite[Volume 1, p.\ 118]{Zygmund} shows that 
\[
\|g^{(\theta,K)}(z)\|_{p,\alpha} = \frac{1}{4\delta^2} 
\left\| \int_0^{2\delta} f^{(\theta,K)}(ze^{it}) + f^{(\theta,K)}(ze^{-it})
                - 2f^{(\theta,K)}(z) (2\delta - t) \, dt \right\|_{p,\alpha}
\]
which shows that $\|g^{(\theta,K)}(z)\|_{p,\alpha} \leq (1/2) M \delta$.  
Applying Theorem \ref{thm:approx_nderiv} to 
$g$ and $f_{\delta \delta}$ and setting 
$\delta = 2\pi/n$ now yields the result. 
\end{proof}

\section{Approximation of Extremal Functions by Polynomials
in the Bergman Norm}

We now discuss extremal problems restricted to the space of 
polynomials of degree $n$.  Let $F_n$ denote the extremal polynomial of 
degree $n$, for the extremal problem of maximizing $\Rp \phi(f)$ 
where $f$ ranges over all 
polynomials of degree at most $n$ with norm $1$. 
We will need the following theorem from \cite{tjf:holder-bergman}. 
\begin{theorem}\label{thm:ext-regularity}
Suppose that $k \in \Lambda^*_{\beta,A^{q}_\alpha}$, and let $F$ be the extremal 
function for $k$.
Then if $2 \le p < \infty$ we have 
$F \in \Lambda^*_{\beta/p, A^p_\alpha}$ 
while if $1 < p \le 2$ we have
$F \in \Lambda_{\beta/2, A^p_\alpha}$.  

Furthermore, suppose that 
$\int_{\mathbb{D}} F \conj{k} \, dA_{\alpha} = 1$ and 
$\|k(e^{it}\cdot) + k(e^{-it} \cdot) - 2k(\cdot)\|_{q,\alpha} 
\leq B|t|^{\beta}$. 
If $p \ge 2$ then 
$\|F\|_{\Lambda^*, \beta/p, A^p_\alpha} \le 
2p^{1/p} (B/2)^{1/p} \leq
2e^{1/e} (B/2)^{1/p}$ whereas if 
$1 < p < 2$ then 
$\|F\|_{\Lambda^*, \beta/2, A^p_\alpha} \le 
2(p-1)^{-1/2}B^{1/2}$.
\end{theorem}

The space of polynomials of degree 
$n$ is isomorphic with $\mathbb{R}^{2n+2}$.  
The set of all $x \in \mathbb{R}^{2n + 2}$ for which the 
corresponding polynomial has norm of at most $1$ is a 
convex set.  
Thus, the extremal 
problem for finding $F_n$ can be thought of as a problem of maximizing 
a (real) linear functional over a convex set in $\mathbb{R}^{2n + 2}$.  
This is a convex optimization problem, and many algorithms for 
approximating the solution are known. 

We first discuss a worst case rate of convergence of 
$F_n$ to $F$ in the Bergman space norm.  

\begin{theorem}
Let $F$ be the extremal function for $\phi$ and let $F_n$ be the 
extremal polynomial of degree $n$, when the problem is posed over 
polynomials of degree $n$. Suppose 
$k \in \Lambda^*_{\beta, A^q_\alpha}.$
Then for $p < 2$ we have
$\|F - F_n\|_{p,\alpha} = O(n^{-\beta/4})$.  Similarly if $p > 2$ we have 
$\|F - F_n\|_{p,\alpha} = O(n^{-\beta/p^2})$.  

More precisely, 
for $p <2$ and $0 < \beta < 2,$
\[ \|F-F_n\|_{p,\alpha} \leq 4(p-1)^{-3/4} A_{\beta/2}^{1/2} \|k\|_{\Lambda^*, \beta, A^q_\alpha}^{1/4} n^{-\beta/4}
;\]
for $p <2$ and $\beta=2$
\[ \|F-F_n\|_{p,\alpha} \leq 4(p-1)^{-3/4} \widetilde{B_0}^{1/2} \|k\|_{\Lambda^*, \beta, A^q_\alpha}^{1/4} n^{-\beta/4}
;\]
for $p > 2$ and $0 < \beta \leq 2,$
\[ \|F-F_n\|_{p,\alpha} \leq 2^{1+1/p-1/p^2} p^{1/p+1/p^2}  A_{\beta/p}^{1/p}  
   \|k\|_{\Lambda^*, \beta, A^q_\alpha}^{1/p} n^{-\beta/p}
.\]
\end{theorem}
\begin{proof}
Let $\|\phi\|$ denote $\|\phi\|_{(A^p_\alpha)^*}$. 
The argument in \cite[Theorem 4.1]{tjf1} shows that, if 
$T_n$ is the best approximate of $F$ of degree $n$ and 
$E_n^{p,\alpha} < \delta$ and $\widetilde{T}_n = T_n/\|T_n\|_{p,\alpha}$, then 
$\Rp \phi(\widetilde{T}_n) \geq \frac{1-\delta}{1+\delta} \|\phi\|$. 
This also shows that 
$\Rp \phi(F_n) \geq \frac{1-\delta}{1+\delta} \|\phi\|$.  
Thus 
\[
\phi((F_n + F)/2) \geq \|\phi\| 
   \left( \frac{1}{2} + \frac{1-\delta}{2(1+\delta)} \right). 
\]
Therefore $(1/2)\|F_n + F \| \geq \frac{1}{2} + 
      \frac{1-\delta}{2(1+\delta)} \geq 
1-\delta$.  This shows that 
$\|F_n - F\| \leq \sqrt{\frac{8}{p-1}} \delta^{1/2}$ for 
$p < 2$ and 
$\|F_n - F\| \leq 2 p^{1/p} \delta^{1/p}$ for 
$p > 2$.  

\end{proof}

The convergence rate in the previous theorem may be slow, especially 
for large $p$.  However, this is a worst case scenario and a given 
$F_n$ may be more accurate than this predicts.  The following theorem 
give a way to bound the distance of a given function $g$ from $F$ in 
terms of the distance from $\bp_\alpha(|F|^p/\overline{F})$ to 
$\bp_\alpha(|g|^p/\overline{g})$. 
An advantage of the theorem is that 
it applies to any $A^p_\alpha$ function $g$, so we can directly apply it 
to an approximation of $F_n$, and not just $F_n$ itself. 
In the theorem statement, $\bp_\alpha$ denotes the Bergman projection for 
$A^p_\alpha$, which is the orthogonal projection from 
$L^2_\alpha$ onto $A^2_\alpha$.  
Also $|F|^p/\overline{F} = F^{p/2} \overline{F^{(p/2)-1}} = 
|F|^{p-1} \sgn F$ should be interpreted to equal $0$ when $F$ has a zero. 
It is known that 
$P_\alpha$ is bounded from $L^p_\alpha$ to $A^p_\alpha$ for 
$1 < p < \infty$ (see \cite{Zhu_Ap}). 

\begin{lemma}
Suppose that $F_1$ and $F_2$ are the $A^p_\alpha$ extremal functions for 
$\phi_1$ and $\phi_2$ respectively
Suppose that $\|\phi_1\| = \|\phi_2\| = 1$ and 
$\|\phi_1 - \phi_2\| < \delta$.  Then for $p > 2$ 
\[
\|F_1 - F_2\| < 2^{1-(1/p)} p^{1/p} \delta^{1/p};
\]
for $p < 2$ 
\[
\|F_1 - F_2\| < 2 (p-1)^{-1/2} \delta^{1/2}.
\]
\end{lemma}
\begin{proof}
Note that 
\[
\begin{split}
|\phi_1(F_1) + \phi_1(F_2)| & \geq
|\phi_1(F_1) + \phi_2(F_2)| + |(\phi_1 - \phi_2)(F_2)| = 
\|\phi_1\| + \|\phi_2\| - \delta \\ &\geq 
2 - \delta.
\end{split}
\]

This implies that 
\[
\left\| \frac{F_1 + F_2}{2} \right\| > 1 - \frac{\delta}{2}.
\]
The result now follows by Lemma \ref{lemma:sharp_uniform_convexity}
\end{proof}

It is known that if $k$ is a positive scalar multiple of 
$\bp_\alpha(|F|^p/\overline{F})$, where $F$ has unit norm, then
$F$ is the extremal function for $k$.  
Since 
$\int_{\mathbb{D}} F \overline{\bp_\alpha(|F|^p/\overline{F})} \, dA = 
\int_{\mathbb{D}} F \overline{|F|^p/\overline{F}} \, dA = 1$, we see that 
if $k$ is scaled so that $\int_{\mathbb{D}} F \overline{k} \, dA_\alpha = 1$, 
then $k = \bp_\alpha(|F|^p/\overline{F})$.

\begin{theorem}\label{thm:ferrork}
Let $k \in A^q_\alpha$, and let $F$ be the extremal function for 
$k$.  Let $\widehat{k}$ be any positive scalar multiple of 
$k$ (so that $\widehat{k}$ also has $F$ as extremal function.) 
Let $G \in A^p_\alpha$ and suppose that for some 
$\delta$ such that $0 < \delta < 1$ the inequality
\[
\|\bp_\alpha(|G|^p/\overline{G}) - \widehat{k} )\|_{A^q_\alpha} 
< \delta
\]
is satisfied.  Then for $p > 2$,
\[
\|F - G\| < 2 p^{1/p} \delta^{1/p}
\]
and for $p < 2$
\[
\|F_1 - F_2\| < 2\sqrt{2} (p-1)^{-1/2} \delta^{1/2}.
\]

\end{theorem}
\begin{proof}
Let $\psi$ be the functional of unit norm for which $G$ is the extremal 
function.  Then $\psi$ has kernel 
$\bp_\alpha(|G|^p/\overline{G})$ and $\|\psi\|=1$. 
Let $\phi$ be the functional with kernel $\widehat{k}$. 
We then have 
\[
\|\phi - \psi\|_{(A^p_\alpha)^*} \leq 
\|\bp_\alpha(|G|^p/\overline{G}) - \widehat{k} )\|_{A^q_\alpha} 
 < \delta. 
\]
 This implies 
that $1-\delta < \|\phi\| < 1 + \delta$.  
Let $\widetilde{\phi} = \phi / \|\phi\|$. 
Then
\(
\| \phi - \widetilde{\phi} \| < \delta
\)
and thus $\|\widetilde{\phi} - \psi\| < 2 \delta$.  
The conclusion now follows from the previous lemma.
\end{proof}

\section{Approximation of Extremal Functions 
by Polynomials in the Supremum Norm}

We now discuss how to use the results in the previous section to bound 
the distance from a given function to $F$ in the supremum norm.  
We will use the following theorem found in
\cite[Corollary 4.3]{tjf:holder-bergman}.
The proof of this theorem shows that the same results hold if 
$F$ is replaced by $F_n$.  However, we may need to multiply $k$ by 
a positive scalar constant greater than $1$ so that the condition 
$\int_{\mathbb{D}} F_n \conj{k} \, dA_{\alpha} = 1$ holds.  

\begin{theorem}\label{thm:holder_bounds}
Let $1 < p < \infty$ and let $p$ and $q$ be conjugate exponents. 
Suppose $k \in \Lambda^*_{2,A^q_\alpha}$ and that 
$\int_{\mathbb{D}} F_n \conj{k} \, dA_{\alpha} \geq 1$.   
If $2 \le p < \infty$ and  
$-1 < \alpha < 0$, then $f$ has H\"{o}lder continuous boundary values.
If $1 < p < 2$ and $-1 < \alpha < p-2$, the same conclusion holds. 

Let $B = \|k\|_{\Lambda^*,2, A^p_\alpha}$.  
For $p > 2$, 
The H\"{o}lder exponent is $-\alpha/p$.  The H\"{o}lder constant 
is bounded above by 
\[
2p^{1/p}(B/2)^{1/p} \cdot 383\left(1-\frac{2}{p}\right)^{-1} 
  \cdot 2\left(\frac{\Gamma(q-1)}{\Gamma(q/2)^2}\right)^{1/q} 
\cdot \left(1 - \frac{2p}{\alpha}\right) 
\]

For $p < 2$, if we let $\eta$ be any number greater than $0$, then 
the H\"{o}lder exponent can be taken to be $1-2/p-\alpha/p-\eta$ 
(if the indicated exponent is positive).  
The H\"{o}lder constant is bounded above by 
\[
\begin{split}
2(p-1)^{-1/2}B^{1/2} &\cdot 192\left(1-\frac{2}{p}\right)^{-1} 
  \cdot 2\left(\frac{\Gamma(q-1)}{\Gamma(q/2)^2}\right)^{1/q} 
\\
&\cdot \left(1 - \frac{2}{1-2/p-\alpha/p-\eta}\right)
\end{split}
\]
\end{theorem}
For ease of notation, we will call the H\"{o}lder exponent 
$\beta(p, \alpha)$ for $p > 2$ and $\beta(p, \alpha, \eta)$ 
for $p < 2$.  We will denote the constant by $C(B,p,\alpha)$ and 
$C(B,p,\alpha,\eta)$ respectively.    
For $p > 2$ if we refer to 
$\beta(p, \alpha, \eta)$ and $C(p, \alpha, \eta)$, we mean 
$\beta(p, \alpha)$ and $C(p, \alpha)$ respectively. 

Since $\|F\| = 1$, it follows that $|F(0)| < 1$.  
Thus the preceding estimate can be used to bound 
$\|F\|_\infty$.  However, the estimates do not allow one to 
conclude directly that $\|F - F_n\|_\infty$ must be small for large 
$n$.  The following theorem remedies this situation.  It says that 
if a function is H\"{o}lder continuous (with control on the exponent 
and size of the constant) and the function has small 
$L^p_\alpha$ norm, then its uniform norm cannot be too large.

\begin{theorem}\label{thm:holder_bergman_uniform_bound}
Let $\epsilon > 0$ and $0 < \beta \leq 1$ be given.  
Suppose that $f \in L^p_\alpha(\mathbb{D})$ and that for some $C > 0$ we have 
$|f(z) - f(w)| \leq C |z-w|^\beta$ for every $z, w \in \mathbb{D}$.  Then 
there exists a $\delta > 0$ such that if $\|f\|_{p,\alpha} < \delta$ then 
$\|f\|_{\infty} < \epsilon$.  In fact, we may take $\delta$ to be
\begin{equation}\label{eq:delta}
\left(\frac{(\alpha + 1)\pi}{4}\right)^{1/p} B(2/\beta, p+1)^{1/p} 
               C^{-2/(\beta p)}\epsilon^{1 + 2/(\beta p)}
\end{equation}
as long as $\epsilon < 2^{\beta/2} C$. Here $B(x,y)$ is the Beta function. 
\end{theorem}
For ease of notation we will denote the $\delta$ in the theorem by 
$\delta(\epsilon; C, \beta, p, \alpha)$.  
We let $\epsilon(\delta; C, \beta, p, \alpha)$ denote the inverse function 
of $\delta(\epsilon) = \delta(\epsilon; C, \beta, p, \alpha)$.  

\begin{proof}
Suppose that $|f(z_0)|  > b > 0$.  Then 
$|f(z)| > b - C |z-z_0|^\beta$ for $0 \leq |z-z_0| \leq r_0$, where 
$r_0 = (b/C)^{1/\beta}$.
So 
\[
\|f\|_p^p > 
 (\alpha + 1)\int\limits_{\substack{z \in \mathbb{D}\\|z-z_0|<r_0}} 
   (b - C|z-z_0|^\beta)^p \, (1-|z|^2)^\alpha \, dA(z). 
\]
Now for fixed $b$, the quantity on the right is a continuous function of 
$z_0$ for $z_0 \in \overline{\mathbb{D}}$, and thus has a minimum; 
call the minimum $\delta(b)^p$.  Then if 
$|f(z_0)| \geq b$ we have $\|f\|_p \geq \delta(b)$.  So if 
$\|f\|_p < \delta(\epsilon)$ we have $\|f\|_\infty < \epsilon$. 

We may estimate $\delta(\epsilon)$ for $r_0 < \sqrt{2}$ by noting that 
in this case the region $\mathbb{D} \cap \{z: |z-z_0|<r_0\}$ contains at least 
a quarter sector of the disc $\{z:|z-z_0| < r_0\}$, so 
\[
\delta(b)^p \geq \frac{\alpha + 1}{4} \int_0^{r_0} 
  \int_0^{\pi / 2} (b-Cr^\beta)^p 2r \, dr = 
\frac{(\alpha + 1)\pi}{4} b^{p+(2/\beta)} C^{-2/\beta} B(2/\beta, p+1) 
\]
where $B(x,y)$ is the beta function. 
\end{proof}

We may also prove the following theorem.  It will not be used in the sequel, 
but we include it for completeness.  
\begin{theorem}
Let $\epsilon > 0$ and $0 <\gamma < \beta \leq 1$ be given.  
Suppose that $f \in L^p(\mathbb{D})$ and that for some $C > 0$ we have 
$|f(z) - f(w)| \leq C |z-w|^\beta$ for every $z, w \in \mathbb{D}$.  Then 
there exists a $\delta > 0$ such that if $\|f\|_p < \delta$ then 
$|f(z)-f(w)| < \epsilon |z-w|^\gamma$.  
\end{theorem}
\begin{proof}
Suppose $|f(z) - f(w)| \geq \epsilon |z-w|^{\gamma}$ 
for some $z$ and $w$.  Since 
$|f(z) - f(w)| < C |z-w|^{\beta}$ we have 
$|z-w|^{\beta-\gamma} > \epsilon / C$.  Thus 
$|z-w|^{\gamma} > (\epsilon/C)^{\gamma/(\beta-\gamma)}$, and so 
$|f(z) - f(w)| > \epsilon^{\beta/(\beta-\gamma)} C^{-\gamma/(\beta - \gamma)}$.  
But this contradicts the previous theorem if $\delta$ is small enough.
\end{proof}

\begin{theorem}\label{thm:extinfinity_error}
Let $1 < p < \infty$ and let $p$ and $q$ be conjugate exponents. 
Suppose $k \in \Lambda^*_{2,A^q_\alpha}$.  
If $2 \le p < \infty$ let 
$-1 < \alpha < 0$. 
If $1 < p < 2$ let $-1 < \alpha < p-2$.
Suppose that 
$\int_{\mathbb{D}} F \conj{k} \, dA_{\alpha} > 1$.  
Then if $\|F-F_n\|_{A^{p}_\alpha} < \delta$ then 
$\|F-F_n\|_\infty < 
\epsilon\left(\delta; C, \beta, p, \alpha\right)$, where 
$\beta = \beta(p, \alpha, \eta)$ and 
\[
C = C(\|k\|_{\Lambda^*, 2, A^q_\beta}, p, \alpha, \eta) +
     C((1-\delta)^{-1} \|k\|_{\Lambda^*, 2, A^q_\beta}, p, \alpha, \eta).
\]
\end{theorem}
\begin{proof}
This follows from Theorems 
\ref{thm:holder_bounds} and 
\ref{thm:holder_bergman_uniform_bound}. 
We use the fact that Theorem 
\ref{thm:holder_bounds} applies to 
$F_n$ if $k$ is first multiplied by 
$1/(1-\delta)$, which ensures that the condition 
$\int_{\mathbb{D}} F_n \overline{k} \, dA_\alpha \geq 1$
holds. 
\end{proof}
\section{Approximation of Extremal Functions for Even $p$}

We will give an example of approximating an 
extremal function. 
The case where $p$ is even is in some ways easier than other 
cases since then we can explicitly compute 
$\bp(|f|^p/\overline{f}) = \bp(f^{p/2} \overline{f}^{p/2-1})$ when $f$ is a polynomial, 
due to the fact that $f^{p/2}$ and $f^{p/2-1}$ are polynomials,  
so our example will involve this case. 

Define 
\[
\gamma(n,\alpha) = \|z^n\|^2_{A^2_\alpha} = 
(\alpha + 1) B(n+1, \alpha + 1) = 
\frac{\Gamma(\alpha + 2) \Gamma(n+1)}{\Gamma(n + \alpha + 2)}.
\]

Then 
\begin{equation}\label{eq:monnorm}
\bp_\alpha(z^m \overline{z}^n) = 
 \begin{cases}
\frac{\gamma(m, \alpha)}{\gamma(m-n,\alpha)} z^{m-n} 
   &\text{ if $m \geq n$} \\
0 &\text{ if $m < n$}
\end{cases}
\end{equation}
(see \cite[Section 1.1]{Zhu_Ap}). 

\begin{example} 
Let us approximate the solution to the problem of maximizing the (real part of) the functional
$f \mapsto a_0 + a_1 + a_2$, where the $a_n$ are the 
Taylor series coefficients of $f$ about $0$, and where 
$p = 4$ and $\alpha = -1/2$ (and where $f$ has 
unit norm). 
Then $k = a_0 + a_1/\gamma(1,-1/2) + a_2/\gamma(2,-1/2) = 
1 + (3/2)z + (15/8)z^2$.  

This problem is made simpler because the uniqueness of 
$F$ implies that it must have real coefficients.  Let us 
take the approximation of degree $N = 20$.  We thus seek to 
maximize $a_0 + a_1 + a_2$ subject to the constraint 
$ \|f\|^4_{4,-1/2} = \|f^2\|^2_{2,-1/2} \leq  1$, i.e.
\[
\sum_{n=0}^N \left(\sum_{m=0}^n 
   a_m a_{n-m} \right)^2 \gamma(n,-1/2) \leq 1.
\]
This is a convex optimization problem, and we are aided 
by the fact that any local maximum must be a global 
maximum, since if $F$ is any local maximum (necessarily of norm $1$) 
then a variational argument similar to the one in the proof of 
\cite[Chapter 5, Lemma 2]{D_Ap} shows that the 
$\bp_{\alpha, 20}(|F|^p/\overline{F})$ is a scalar multiple of $k$,
and thus $F$ is the extremal function 
(see \cite[p.\ 55]{Shapiro_Approx}). 
Here we let $\bp_{\alpha, 20}$ denote the orthogonal projection from 
$L^p_\alpha(\mathbb{D})$ onto the subspace of $A^p_\alpha$ of 
consisting of polynomials of degree at most $20$.  

Using Mathematica (for example) to approximate a solution 
yields a maximum functional value of $1.78785$ and 
\[
\begin{split}
F_{20} = 
0.431458 &+ 0.496144 x + 0.860246 x^2 - 0.341597 x^3 - 
0.0225992 x^4 + 
 0.110915 x^5 \\
&{}- 0.0520239 x^6 - 0.00952809 x^7 + 0.0235908 x^8  + 
\cdots + 
- 0.000599527 x^{15}
\end{split}
\]
All of the omitted terms have coefficients of less than 
$1/100$.  
If we compute 
$\widetilde{k} = \bp_\alpha(|F_{20}|^4/\overline{F_{20}})$, 
we find that it is 
\[
.559332 + 0.838998 z + 1.04875 z^2 + \cdots + 
1.55098 \cdot 10^{-9} z^{40}.
\]
All of the omitted terms have coefficients of at most
$.0001$. 
We must now find a multiple of $k$ close to 
$\widetilde{k}$.  We could find the closest one 
as an optimization 
problem, but we will choose
$\widehat{k} = .559332 k$ in order to make the first 
coefficients of $\widetilde{k}$ and $\widehat{k}$ match, 
since this is simpler and yields a result close to 
$\widetilde{k}$. 
If we now compute 
$\|\widetilde{k} - \widehat{k}\|_{A^{4/3}_{-1/2}}$, we 
see that it is about $.000018$. Theorem 
\ref{thm:ferrork} shows that
$\|F - F_{20}\|_{4,-1/2}$ is less than $.185$.   
In fact, I suspect 
that the true error is much smaller.  For example, 
$\|F_{25}-F_{20}\|_{4,-1/2} < 1.85 \cdot 10^{-5}$, so 
the true error may be closer to this order of 
magnitude. 

We find that $.559332$ times the sum of the first three 
coefficients of $F$ is bigger than $1$ by about 
$2.5 \cdot 10^{-7}$.  The second $\theta$ derivative of 
$\widehat{k}$ is most $5.034$, so 
$\|\widehat{k}\|_{\Lambda^*,2,A^{4/3}_{-1/2}}$ is at most  
$5.034$.  Thus Theorem 
\ref{thm:extinfinity_error} shows that 
$\|F - F_{20}\|_{\infty} < 5181$.  
Again, I suspect the true error is much smaller.  For 
example, $\|F_{25}-F_{20}\|_\infty < .00006$, and 
the true error may be this order of magnitude.  

It would be interesting to see if the estimates in this paper can be 
substantially improved in order to yield better estimates on 
the approximation of extremal functions in the uniform norm.  
The example above shows that the estimates in the paper are likely too 
large by a substantial margin.  However, the estimates in this paper are the 
only ones known (as far as I know) 
that allow approximation of these extremal functions 
in the uniform norm, 
and they have the advantage of being explicitly computable 
without great difficulty. 
\end{example}

\section{Non-zero Extremal Functions}
The proceeding results can be used to find explicit conditions on 
$k$ that guarantee that $F$ is non-zero.  In 
Theorem \ref{thm:ksector-fnonzero} we give one such result. 

\begin{theorem}
Let $0 < \theta < 2\pi$ and $\theta < 2\pi(p-1)$.  
Suppose that $k \in A^{q}_\alpha$ has range that is a subset of 
the sector 
$-\theta/2 < \arg z < \theta/2$, and that
\(
\|k\|_{q,\alpha} = 1 
\).
Let $F$ be the extremal function for $k$ and let
\[
C_\theta = 2 C_{p,\alpha} 
\left|\sin\left(\frac{(p-2)\theta}{4(p-1)}\right) \right|,
\] where 
$C_{p,\alpha}$ is the bound for the Bergman projection from 
$L^p_\alpha$ onto $A^p_\alpha$.  Then if $p > 2$ we have  
$\|F - k^{1/(p-1)}\|_{p,\alpha} \leq 2p^{1/p} C_\theta^{1/p}$ and 
if $p < 2$ we have 
$\|F - k^{1/(p-1)}\|_{p,\alpha} \leq 2\sqrt{2}(p-1)^{-1/2} C_\theta^{1/2}$.

\end{theorem}
\begin{proof}
Note that 
$G = k^{1/(p-1)}$ is well defined, where we take the 
branch with $1^{1/(p-1)} = 1$.  
Notice that $|G|^{p-1}\sgn G = |k| e^{\arg(k)/(p-1)}$.  
Thus 
\[
|k - |G|^{p-1}\sgn G| = 
 |k|\left|e^{i\arg(k)} - e^{i\arg(k)/(p-1)}\right|
\leq 2|k| 
 \left|\sin\left(\frac{(p-2)\theta}{2(p-1)}\right) \right|
\]
and therefore
\[
\|k - |G|^{p-1}\sgn G\|_{p,\alpha} 
\leq 2\|k\|_{p,\alpha} 
 \left|\sin\left(\frac{(p-2)\theta}{2(p-1)}\right) \right|.
\]
Let $C_{p,\alpha}$ be the bound for the Bergman projection from 
$L^p_\alpha$ onto $A^p_\alpha$.  Then 
\[
\|k - \bp_\alpha(|G|^{p-1} \sgn G)\|_{q,\alpha} \leq 
2 C_{p,\alpha} 
\left|\sin\left(\frac{(p-2)\theta}{4(p-1)}\right) \right|
\|k\|_{q,\alpha}.
\]
since $\bp_\alpha(k) = k$. 
The result now follows from Theorem \ref{thm:ferrork}.
\end{proof}

\begin{theorem}\label{thm:ksector-fnonzero}
Let $0 < d < 1$ and 
$2 \le p < \infty$ and 
$-1 < \alpha < 0$. 
If $1 < p < 2$ also suppose $-1 < \alpha < p-2$.
Let
$\|k\|_{q,\alpha} = 1$ and suppose that 
$\|k\|_{\Lambda^*,2,A^p_\alpha} < B$. 
Then there exists a $\theta > 0$ depending only on 
$d$, $B$, $p$, and $\alpha$ such that 
if the range of $k$ is a subset of 
$\{z: -\theta/2 < \arg z < \theta/ 2 \text{ and } |z| > d\}$ then 
$F$ is non-zero. 
\end{theorem}
\begin{proof}
Let $\theta > 0$ be given.  This $\theta$ will make the conclusion of the
theorem true if 
the assumptions show that $\|F - k^{1/(p-1)}\|_\infty < d^{1/(p-1)}$.
Let $\lambda = d^{1/(p-1)}$.   
For $p < 2$ choose $0< \eta <1 - 2/p - \alpha/p$ and
let $\beta = \beta(p,\alpha, \eta)$; otherwise let 
$\beta = \beta(p, \alpha)$. 

Note that 
by \cite[Theorems 3.1 and 1.2]{tjf:holder-bergman} and 
\cite[Theorems 5.9 and 5.1]{D_Hp}, we have
$k \in C^{2-2/p-\alpha/p} \subset C^{\beta}$ 
with H\"{o}lder constant depending only on $B$, $p$, and 
$\alpha$.  
Since $k$ is bounded away from $0$, we also have that 
$k^{1/(p-1)} \in C^{\beta}$ with constant depending only on 
$B$, $d$, $p$, and $\alpha$. 
Let $D$ be the smallest constant such that 
$|k(z)^{1/(p-1)} - k(w)^{1/(p-1)}| \leq D |z-w|^\beta$ .

By Theorem \ref{thm:holder_bergman_uniform_bound} 
we will be done if we can show that 
\[
\|F - k^{1/(p-1)}\|_{p,\alpha} < \delta 
\] 
where
\( \delta = \delta(\lambda, C + D, \beta, p, \alpha),
\)
where $C = C(B, p,\alpha,\eta)$.  But by the previous theorem, this is 
true if $\theta$ is small enough.  
\end{proof}
Notice that, given $B$, $d$, $\epsilon$, $p$ and $\alpha$, we could if we wish 
calculate an explicit value for $\theta$. 

\providecommand{\bysame}{\leavevmode\hbox to3em{\hrulefill}\thinspace}
\providecommand{\MR}{\relax\ifhmode\unskip\space\fi MR }
\providecommand{\MRhref}[2]{%
  \href{http://www.ams.org/mathscinet-getitem?mr=#1}{#2}
}
\providecommand{\href}[2]{#2}

\end{document}